\documentclass{amsart} 

\usepackage[latin1,utf8]{inputenc}
\usepackage{amsfonts,bbm}
\usepackage{amssymb}
\usepackage{graphicx}
\usepackage{graphics}
\usepackage{lscape}
\usepackage{comment}
\usepackage{mathrsfs}
 
\newcommand{\QI}{q}
\newcommand{\E}{\mathbbm{E}}
\renewcommand{\P}{\mathbbm{P}}
\usepackage{amsmath}

\newtheorem{conjecture}{Conjecture}
\newtheorem{lemma}{Lemma}

\newtheorem{prop}{Proposition}

\newtheorem{corol}{Corollary}

\title[Processor Sharing Queue with batches]{{On the sojourn time of a batch in the $M^{[X]}/M/1$ Processor Sharing Queue}}

\author[F. Guillemin]{Fabrice Guillemin}
\address[F. Guillemin]{Orange Labs Networks Lannion, 2 avenue Pierre Marzin, 22307 Lannion Cedex, France}
\author[A. Simonian]{Alain Simonian}
\author[R. Nasri]{Ridha Nasri}
\address[A. Simonian and R. Nasri]{ Orange Labs, DATA-IA, Orange Gardens, 44 avenue de la République, CS 50010, 92326 Châtillon Cedex,  France}
\author[V. Quintuna]{Veronica Quintuna Rodriguez}
\address[V. Quintuna]{Orange Labs Networks Lannion, 2 avenue Pierre Marzin, 22307 Lannion Cedex, France}
\email{\{first\_name.last\_name\}@orange.com}

\begin{document}

\date{Version of \today}

\begin{abstract}
In this paper, we analyze the sojourn of an entire batch in a processor sharing  $M^{[X]}/M/1$ processor queue, where geometrically distributed  batches arrive according to a Poisson process and jobs require exponential service times. By conditioning on the number of jobs in the systems and the number of jobs in a tagged batch, we establish recurrence relations between conditional sojourn times, which subsequently allow us to derive a partial differential equation for an associated bivariate generating function. This equation involves an unknown generating function, whose coefficients can be computed by solving an infinite lower triangular linear system. Once this unknown function is determined, we  compute the Laplace transform and the mean value of the sojourn time of a batch in the system.

\end{abstract}

\keywords{Batch $M/M/1$ queue; Processor sharing; Sojourn time; Laplace transform; Infinite linear system.}

\maketitle


\section{Introduction}
\label{Sec:Intro}

In this paper, we consider an $M^{[X]}/M/1$ queue, where batches (or bulks) of jobs arrive according to a Poisson process with rate $\lambda$ and individual jobs require exponential service times with mean $1/\mu$; the probability that the  size $B$ of a batch is equal to $b$ is given by $q-b$ (i.e., $\P(B=b)=q_b$) for some sequence of positive real numbers $(q_b)$ such that $\sum_{b=1}^\infty q_b=1$. Even if we establish some properties of the system for an arbitrary batch size, we shall mainly consider the case when $q_b=(1-q)q^{b-1}$ for some $q\in (0,1)$, i.e., batches are geometrically distributed. We shall further assume that the service discipline is Processor Sharing. This means that the service capacity is equally shared among all jobs present in the system. The queuing system under consideration is referred to as $M^{[X]}/M/1$-PS queue.

The number of jobs in the system is well studied in the literature and can be found in standard textbooks \cite{grossharris,Klein0}. The analysis of sojourn time of jobs is much more difficult as it involves complex correlations between the sojourn times of all jobs in the system. The mean waiting time of a job in the system was studied by Kleinrock \emph{et al} in \cite{Klein71} for arbitrary batch size probability distributions. The authors notably established an integral equation for the sojourn time of a tagged job conditioned on the service time of this  job. For the specific case of geometric batch size, the complete distribution of the sojourn time of a job has been derived in \cite{GQS}.

In this paper, we consider the sojourn time $\Omega$ of a batch in the system, i.e., the time elapsed between the arrival of the batch and the departure of the last job of the batch. This quantity appears as relevant performance parameter when considering the execution of batch of jobs in a cloud system (for instance virtualized network functions), see for instance \cite{veronica2}. In this context, the execution time of an entire batch is critical when considering real time virtualized network functions such as in the case of cloud radio access networks \cite{cloudran}. See also \cite{ayesta, bansal} for other applications.

The quantity $\Omega$ is defined as follows: Given a tagged batch with size $B = b$, $b \geqslant 1$, the sojourn time $\Omega$ equals the maximum
\begin{equation}
\Omega = \max_{1 \leqslant k \leqslant b} W_k
\label{defWbar}
\end{equation}
of the sojourn times $W_k$, $1 \leqslant k \leqslant b$, of the jobs of the tagged batch. Throughout this paper, the service rate $\mu$ will be normalized to 1, so that the arrival rate $\lambda$ is set to $\rho$ with 
$\rho \mathbb{E}(B) < 1$; under this condition, the system is stable.

Because of the correlations between the sojourn times of the various jobs of a batch, the study of the random variable $\Omega$ reveals utmost complex. An approximation has been developed in \cite{itc31} by considering the residual busy period after a tagged batch arrival and by assuming that the jobs of the tagged batch leave the system  at random among those jobs of the residual busy period. By labeling jobs from 1 to the number of jobs in the remaining busy period, by sampling a number of jobs equal to the size of the tagged batch and by computing the greatest index of those jobs, we can estimate the departure time of the last job of the batch and then the batch sojourn time. Even if this approximation gives satisfactory results when compared to simulations, the asymptotic behavior of the batch sojourn time is overestimated. The prefactor of the exponential decay is polynomial with order 3/2 while the prefactor of the exponential decay in the case of a job is subexponential. This may indicate that the proposed approximation is too coarse.

The objective of this paper is to compute the Laplace transform of the random variable $\Omega$ as well as its mean value. For this purpose, we adopt the following strategy: We consider a tagged batch arriving in the system and we introduce the random variables $\Omega_{n,b}$, where $n$ is the number of jobs in the system upon the tagged batch arrival and $b$ is the number of jobs in the batch. We first establish an infinite linear system satisfied by the Laplace transforms of the random variables $\Omega_{n,b}$ and then a partial differential equation (PDE) satisfied by a bivariate generating function associated with these Laplace transforms.

The PDE under consideration involves an unknown univariate generating function. The determination of this function can performed by using analyticity conditions at the origin. It turns out that the coefficients of the unknown generating function satisfy an infinite linear lower triangular system, which involves hypergeometric polynomials and which can be solved by using the results in \cite{ridha}. Once the unknown generating function is determined, the Laplace transform as well as the mean value of the random variable $\Omega$ can be computed.

The organization of this paper is as follows: In Section~\ref{Sec:Prob} we establish the recurrence relations  satisfied by the Laplace transforms of the random variables $\Omega_{n,b}$, $n \geq 0, b\geq 1$. By using this differential system, we derive in Section~\ref{Sec:PDE} a PDE satisfied by a bivariate generating function associated with those Laplace transforms when the batch size is geometric. The resolution of this PDE is performed in Section~\ref{Sec:Resol}. The Laplace transform and the mean value of the batch sojourn time $\Omega$ is eventually computed in Section~\ref{Sec:Batch}.  Concluding remarks are presented in Section~\ref{Sec:Conclusion}. The proofs of some technical results are deferred to the Appendix.

\section{Notation and fundamental recurrence relations}
\label{Sec:Prob}

As mentioned in the Introduction, we consider for  $n \geq 0 $ and $b \geq 1$,   the random variable $\Omega_{n,b}$, which is the sojourn time of a tagged batch in the queue, given that 
\begin{itemize}
\item  there are $n \geqslant 0$ jobs in the  queue at the  arrival instant of the tagged batch,
\item and the size of the tagged batch is equal to  $b \geqslant 1$.
\end{itemize}
We denote by $e_{n,b}$ the probability density function of sojourn time $\Omega_{n,b}$, that is, in the sense of distributions,
$$
e_{n,b}(x) = \frac{d}{dx} \mathbb{P}(\Omega_{n,b} \leq x), 
\qquad x \in \mathbb{R}^+.
$$
As $\Omega_{n,b} > 0$ almost surely (since the sojourn time includes the non-zero service times of the jobs of the tagged batch), we note that 
\begin{equation}
e_{n,b}(0) = 0, \quad n \geqslant 0, \; b \geqslant 1.
\label{PW1}
\end{equation}
Finally, define  the Laplace transform
$$e^*_{n,b}(s) \stackrel{def}{=} \mathbb{E}(e^{-s \Omega_{n,b}}). $$
These Laplace transforms can  be used to compute the Laplace transform of the sojourn time $\Omega$ of an arbitrary batch since
$$
\E\left(e^{-s \Omega}\right) =\sum_{n=0}^\infty \sum_{b=1}^\infty e^*_{n,b}(s) \P(N=n)\P(B=b).
$$

In this section, we establish the fundamental recurrence relations satisfied by the Laplace transforms $e^*_{n,b}(s)$ for $n \geq 0$ and $b \geq 1$.

\begin{prop}
{The Laplace transforms $e^*_{n,b}(s)$, $n \geqslant 0$, $b \geqslant 1$, verify the recurrence relations for $n \geq 0 $ and $b \geq 1$}
\begin{multline}
e^*_{n,b}(s) =  \frac{\mathbbm{1}_{\{b\geq 2\}} b e^*_{n,b-1}(s) +\mathbbm{1}_{\{b=1\}}}{(n+b)(s+\rho+1)} + 
\frac{1}{s+\rho+1}\frac{n}{n+b} \, e^*_{n-1,b}(s)   \\ 
+ \frac{\rho}{s + \rho + 1} 
\sum_{m \geqslant 1} q_m \, e^*_{n+m,b}(s), \qquad \mbox{for } \Re(s) \geq 0,
\label{E*}
\end{multline}
with the convention $e^*_{-1,b}(s)\equiv 0 $ for all $b \geq 1$.
\end{prop}

\begin{proof}
Consider a tagged batch arriving at some initial time while the queue contains $N = n \geqslant 0$ jobs, and with size 
$b \geqslant 1$. The variable $\Omega_{n,b}$ then satisfies
\begin{equation}
\Omega_{n,b}\stackrel{d}{=} 
\left\{
\begin{array}{ll}
X_{1+\rho} + \Omega_{n,b-1} \quad \quad \quad \; \; \;  
\mathrm{with \; probability} \quad 
\displaystyle \frac{1}{1+\rho} \times \frac{b}{n+b},
\\ \\
X_{1+\rho} + \Omega_{n-1,b} \quad \quad \quad \; \; \; \mathrm{with \; probability} \quad 
\displaystyle   \frac{1}{1+\rho} \times \frac{n}{n+b},
\\ \\
X_{1+\rho} + \Omega_{n+m,b} \quad \quad \quad \; \; 
\mathrm{with \; probability} \quad 
\displaystyle \frac{\rho}{1+\rho} \times q_m, \; \; m \geqslant 1,
\end{array} 
\right.
\label{OmegaNB}
\end{equation}
where $\stackrel{d}{=}$ denotes equality  in distribution and 
$X_{1+\rho}$ is an exponentially  distributed random variable with mean  $1/(1 + \rho)$. To prove the above equality, observe that after the arrival time of the tagged batch, the next event to occur can be either 
\begin{enumerate}
\item a departure due to the service completion of some job in the queue 
(with probability $1/(1+\rho) $). In this first case, 
   \begin{itemize}
	 \item[-] the probability that the service of a job belonging to the tagged batch is completed is equal to $b/(n+b)$ (since $n$ jobs were present at the arrival time of the tagged batch, which has brought a total number of $b$ jobs), hence $\Omega_{n,b} \stackrel{d}{=} X_{1+\rho} + \Omega_{n,b-1}$;  
 
	 \item[-] the probability that this service completion does not occur for any job of the tagged  batch  equals $n/(n+b)$ and we have 
$\Omega_{n,b} \stackrel{d}{=} X_{1+\rho} + \Omega_{n-1,b}$;
   \end{itemize}

\item or the arrival of new batch (with probability 
$ \rho/(1 + \rho)$) with some size $m \geqslant 1$ 
(with probability $q_m$). In this case, the corresponding sojourn time of the tagged batch 
 equals $\Omega_{n,b} \stackrel{d}{= } X_{1+\rho} + \Omega_{n+m,b}$, due to the memory-less property for the service times of all jobs. 
\end{enumerate}

Taking Laplace transforms in  Equation~(\ref{OmegaNB}) immediately yields Equation~\eqref{E*}.
\end{proof}

An explicit solution to the infinite system~(\ref{E*}) does not seem affordable for any distribution $(q_b, b \geqslant 1)$ of the batch size. In the following, we investigate the case of a geometric distribution of the batch size, that is $q_b= (1-q) q^{b-1}$ for some $q\in (0,1)$ such that $1-\rho-q>0$ to ensure the stability of the system. 
 

\section{Partial differential equation for geometric distribution of the batch size}
\label{Sec:PDE}

Let $\mathbb{D} = \{u \in \mathbb{C}, \, \vert u \vert < 1\}$  denote the unit disk in the complex plane and consider the bivariate generating function fo
\begin{equation}
  E(s;u,v) = \sum_{n=0}^\infty\sum_{b=1}^\infty e^*_{n,b}(s) u^n v^b  ,
\end{equation}
defined for $(u,v)\in \mathbb{D}^2$ and $\Re(s) \geq 0$, since by definition $e^*_{n,b}(s) \leq 1$.

Define the quadratic polynomial 
\begin{equation}
P(s;u) = u^2 - (s + 1 + \varrho + \QI)u + s \QI + \varrho + \QI
\label{defPTheta}
\end{equation}
in variable $u$. Recall (\cite{GQS}, Section 4.1) that $P(s,\cdot)$ has two roots $U^-(s)$ and $U^+(s)$, which verify the inequalities 
\begin{equation}
\forall \; s > 0, \qquad q < U^-(s) < 1 < U^+(s).
\label{INEQ-U}
\end{equation}

Furthermore, we consider the function change $E \mapsto F$ where $F$ is defined by
\begin{equation}
F(s;u,v) = 
\left\{
\begin{array}{ll}
\displaystyle \frac{E(s;u,v) - E(s;\QI,v)}{u-\QI}, & \Re(s) \geqslant 0, \; u \in \mathbb{D} \setminus \{\QI\}, \; v \in \mathbb{C},
\\ \\
\displaystyle \displaystyle \frac{\partial E}{\partial u}(s;\QI,v), &
\Re(s) s \geqslant 0, \; u = \QI,  \; v \in \mathbb{C}.
\end{array} \right.
\label{defFTheta}
\end{equation}
By definition, function $F$ is clearly analytic in 
$\{s \; \vert \; s > 0\} \times \mathbb{D} \times \mathbb{D}$ and it is obviously equivalent to determine either function $E$ or $F$. 

As detailed in the following, it proves that $E$ verifies a PDE involving coefficients with polar singularities, which conveniently cancel out by considering the new function $F$. This can be stated as follows; for the sake of alleviating the notation, we omit the Laplace variable $s$ as argument of functions in the following.

\begin{prop}
{Function $F$ defined in Equation~(\ref{defFTheta}) verifies the linear partial differential equation}  for $(u,v) \in \mathbb{D} \times \mathbb{D}$ (and $\Re(s) \geq 0$)
\begin{multline}
 u P(u) \, \frac{\partial F}{\partial u}(u,v) + 
v \left [ \rho(1-\QI) - (s+1+\rho-v)(u-\QI) \right ] \, 
\frac{\partial F}{\partial v}(u,v) \\
+ \; \left [ u(u-s-1-\rho) + (u-\QI)(u+v) \right ] \ F(u,v)  +  L(u,v) = 0 
\label{PDEF0}
\end{multline}
{with polynomial $P$ introduced in (\ref{defPTheta}), and}
\begin{equation}
L(u,v) =L_0(u,v) + (u+v) \, E(\QI,v) -v (s+1+\rho-v)  
\frac{\partial E}{\partial v}(\QI,v),
\label{defL}
\end{equation}
{where}
$$
L_0(u,v) = \frac{v}{(1-u)}.
$$
\label{Corollary1}
\end{prop}

\begin{proof}
Multiplying Equation~\eqref{E*} by $(s+\rho+1)(n+b)u^n v^b$ and summing up for $n$ ranging from 0 to infinity and $b$ from 1 to infinity, we have
\begin{multline*}
(s+\rho+1) u \frac{\partial E}{\partial u} + (s+\rho+1) v \frac{\partial E}{\partial v} = \sum_{n=0}^\infty \sum_{b=1}^\infty \mathbbm{1}_{\{ b \geq 2\}}b e^*_{n,b-1} u^nv^b +\mathbbm{1}_{\{ b=1\}} u^n v \\
+  \sum_{n=0}^\infty \sum_{b=1}^\infty n  e^*_{n-1,b} u^n v^b  + \rho(1-q)  \sum_{n=0}^\infty \sum_{b=1}^\infty \sum_{m=1}^\infty q^{m-1} (n+b) e^*_{n+m,b} u^n v^b .
\end{multline*}
The first term on the r.h.s. of the above equation is equal to 
$$
v^2 \frac{\partial E}{\partial v} +  v E  +\frac{v}{1-u}
$$
and the second term to 
$$
u^2 \frac{\partial E}{\partial u} + u E.
$$

The third term can be written as
$$
\rho(1-q) \sum_{b=1}^\infty \sum_{m=1}^\infty \sum_{n=0}^{m-1} (n+b) q^{m-1-n} u^n e^*_{m,b} v^b. 
$$
By using the fact that
\begin{equation}
\label{idtech}
\sum_{n=0}^{m-1}(n+b)q^{m-1-n} u^n = b \frac{u^m - q^m}{u-q} + m \frac{u^m}{u-q} -u \frac{u^m-q^m}{(u-q)^2},
\end{equation}
we deduce that the third term is equal to 
\begin{multline*}
\rho(1-q) \frac{v}{u-q}\left( \frac{\partial E}{\partial v}(u,v)-  \frac{\partial E}{\partial v}(q,v) \right) + \rho(1-q) \frac{u}{u-q}\frac{\partial E}{\partial u}(u,v) \\
-\rho(1-q) u \frac{E(u,v)-E(q,v)}{(u-q)^2}.
\end{multline*}
We thus obtain the following PDE satisfied by $E(u,v)$
\begin{multline*}
\frac{u P(u)}{u-q}\frac{\partial E}{\partial u}(u,v)  + v (v -(s+\rho+1))   \frac{\partial E}{\partial v} 
 +  \frac{\rho(1-q)v}{u-q}\left( \frac{\partial E}{\partial v}(u,v) -  \frac{\partial E}{\partial v}(q,v) \right)   \\ -\rho(1-q) u \frac{E(u,v)-E(q,v)}{(u-q)^2}  +(u+v)E  
 = -\frac{v}{1-u}  
 \end{multline*}

Introducing function $F(u,v)$ in the above equation, we obtain Equation~\eqref{PDEF0} by simple algebraic manipulation.
\end{proof}

At this stage, we can successively note that

\begin{itemize}
\item[\textbf{a)}] Equation~(\ref{PDEF0}) for $F$ is of order 1 and linear 
\cite[Lecture 1, Section 1.2]{ARN15} with smooth polynomial coefficients in both variables $u$ and $v$;

\item[\textbf{b)}] the last term $L(u,v)$ in (\ref{PDEF0}) involves the unknown function $E$ along with its derivative $\partial E/\partial v$ on the line $u = \QI$. 
\end{itemize}
To completely determine the function $F(u,v)$ and then $E(u,v)$, we have to compute the unknown function $E(q,v)$.

Considering $E(q,v)$ and then $L(u,v)$ as known, Equation~\eqref{PDEF0} can be integrated by using the method of characteristic curves applied in the next Section. Before addressing this integration, another simple variable change will  enable us to transform the quasi-linear Equation~\eqref{PDEF0} into a  simpler linear equation; the proof is purely computational  and hence omitted. 

\begin{corol}
{Let}
\begin{equation}
\Phi(u,v) = P(u) (1-v)  F(u,uv), 
\qquad (u,v) \in \mathbb{D}^2
\label{defPhi}
\end{equation}
{with polynomial $P$ introduced in (\ref{defPTheta}). Then function 
$\Phi$ satisfies the inhomogeneous linear PDE}
\begin{equation}
\frac{\partial \Phi}{\partial u} -  
\left \{ \frac{(u - \QI)}{P(u)} \right\} v(1-v) \frac{\partial \Phi}{\partial v} + 
\mathcal{L}(u,v) = 0,
\label{PDEF0bis}
\end{equation}
{where}
\begin{multline}
    \label{mathcalL}
\mathcal{L}(u,v) = (1-v) \frac{L(u,uv)}{u}  \\
=(1-v)\left(\frac{v}{1-u}+(1+v)E(q,uv)-v(s+1+\rho-uv) \frac{\partial E}{\partial v}(q,uv)  \right).
\end{multline}
\label{C1}
\end{corol}

\section{Determination of function $F$}
\label{Sec:Resol}

\subsection{Resolution of the partial differential equation}


To determine function $F$, we first solve the first order PDE~\eqref{PDEF0bis} by the method of characteristics curves.

\begin{lemma}
The characteristic curve $(t,Z(u,v;t))$ associated with the PDE ~\eqref{PDEF0bis} passing through the point $(u,v)$ is given by
\begin{equation}
\label{charcurve}
    Z(u,v;t)=  \frac{v \frac{{R}(t)}{R(u)}}{(1-v) + v \frac{{R}(t)}{R(u)}},
\end{equation}
where 
\begin{equation}
R(t) = \left(1-\frac{t}{U^-}\right)^{C^--1}  \left(1-\frac{t}{U^+}\right)^{C^+-1} 
    \label{defR}
\end{equation}
with
\begin{equation}
C^+= - \frac{U^- - \QI}{U^+  - U^-}<0,
\quad
C^-= 1-C^+=- \frac{U^+ - \QI}{U^- -U^+}>1,
\label{C+-}
\end{equation}
$U^-$ and $U^+$ being the roots of the polynomial $P(u)$ satisfying inequality~\eqref{INEQ-U}.
\end{lemma}

\begin{proof}
A characteristic curve $(t,z)$ associated with the PDE~\eqref{PDEF0bis} satisfies the ordinary differential equation
$$
\frac{dz}{z(1-z)}= -\frac{(t-q)dt}{P(t)}
$$
and then
$$
\frac{1}{z(1-z)}\frac{dz}{dt} =\frac{q-t}{P(t)} = \frac{C^+-1}{t-U^+}+\frac{C^--1}{t-U^-}.
$$
This implies that 
$$
\frac{z}{1-z}= \kappa R(t),
$$
where $R(t)$ defined by Equation~\eqref{defR} satisfies
\begin{equation}
    \label{derivR}
\frac{dR}{dt} = \frac{q-t}{P(t)} R(t),
\end{equation}
and where $\kappa$ is an integration constant. A characteristic curve is then given by
$$
z(t) = \frac{\kappa R(t)}{1+\kappa R(t)}. 
$$

If we want the characteristic to pass through the point $(u,v)$ then
$$
\kappa = \frac{v}{(1-v)R(u)}
$$
and Equation~\eqref{charcurve} follows.
\end{proof}

Before proceeding with the determination of the function $\Phi(u,v)$, it is worth noting that the  function $u \to R(u)/R(u_0)$ for $u_0\in \mathbb{D}\setminus \{U^-\}$ is analytic on the cut disk $\mathbb{D} \setminus r_{u_0}$, where  $r_{u_0} = \{t: \frac{t-U^-}{u_0-U^-}\leq 0\}$ is the half line issued from $U^-$ and having the same direction as $[u_0,U^-]$; see \cite{Flatto,GQS} for details.

Using the characteristic curves specified above, we immediately obtain the following corollary.
\begin{corol}
The function $\Phi(u,v)$ is given by for $0<|u|<1$ and $|v|<1$
\begin{equation}
    \label{defPhiexp}
    \Phi(u,v) = \int_u^{U^-} (1-Z(u,v;\xi)) L\left(\xi,  \xi Z(u,v;\xi)\right) \frac{d\xi}{\xi},
\end{equation}
where the function $L(u,v)$ is defined by Equation~\eqref{defL} and the function $Z(u,v;t)$ by Equation~\eqref{charcurve}.
\end{corol}
\begin{proof}
Along a characteristic curve (say, passing through the point $(u,v)$), the function $\Phi(t,Z(u,v;t))$ satisfies 
$$
\frac{d\Phi}{dt}(t,Z(u,v;t)) = -\mathcal{L}(t,Z(u,v; t)),
$$
where $\mathcal{L}(u,v)$ is defined by Equation~\eqref{mathcalL} and then 
$$
\Phi(t,Z(u,v;t)) = \int_{t}^{U^-} \mathcal{L}(\xi,Z(u,v;\xi))d\xi+ \kappa
$$
for some integration constant $\kappa$ (depending on $(u,v)$). Taking into account the relation between $F(u,v)$ and $\Phi(u,v)$ and since $F(u,v)$ has to be analytic in $\mathbb{D}^2$ and in particular at point $u=U^-$, we have $\kappa=0$ and the result follows.
\end{proof}

Taking into account the relation between $F(u,v)$ and $\Phi(u,v)$, we deduce the following result.

\begin{lemma}
The function $F(u,v)$ is given for $0<|u|<1$ and $|v|<1$ by
\begin{equation}
    \label{defF}
    F(u,v) = \frac{u}{(u-v)P(u)}\int_u^{U^-}  \left(1-Z\left(u,\frac{v}{u};\xi\right)\right) L\left(\xi,  \xi Z\left(u,\frac{v}{u};\xi\right)\right)  \frac{d\xi}{\xi},
\end{equation}
where the function $Z(u,v;t)$ is defined by Equation~\eqref{charcurve}.
\end{lemma}

It is worth noting that the function $L(u,v)$ depends on $E(q,v)$, which is so far unknown. We give an integral representation of this latter function in the next section.

\subsection{Determination of the unknown function $E(q,v)$}
\subsubsection{Infinite lower triangular linear system for the coefficients}
Let us introduce the coefficients $E_b(q)$ for $b \geq 1$ so that  
$$
E(q,v) = \sum_{b=1}^\infty E_b(q) v^b.
$$
The function $F(u,v)$ is defined so far for $0<|u|<1$ and $v\in \mathbb{D}$. This function shall be analytic for $u=0$ as this function is related to the generating function $E(u,v)$ according to Equation~\eqref{defFTheta}. This requires that  the coefficients $E_b(q)$, $b \geq 1$,  satisfy the following linear system; the proof is given in Appendix~\ref{App3}.

\begin{prop}
\label{CNCL}
The function $F(u,v)$ is analytic in $\mathbb{D}^2$ if and only if for $b \geq 1$
\begin{equation}
\label{syslin}
    \sum_{\ell =1}^b (-1)^\ell \binom{b}{\ell} Q_{b,\ell} E_\ell(q) = b\int_0^{U^-} R(\xi)^b \frac{d \xi}{1-\xi}
\end{equation}
where for $b,\ell \geq 1$
\begin{equation}
\label{defQbl}
    Q_{b,\ell} = \int_0^{U^-} ((\ell-b+1)z -\ell(1+\rho+s)) R(z)^b z^{\ell-1} dz.
\end{equation}
\end{prop}

The coefficients $Q_{b,\ell}$ can be written in terms of Gauss hypergeometric function $F(a,b,c;z)$ \cite{Abramowitz} as 
\begin{multline*}
    Q_{b,\ell}=-\frac{\Gamma(\ell)\Gamma(1-b C^+)}{\Gamma(\ell+1-bC^+)} (U^-)^{\ell+1} \frac{x}{1-x} \\
    \left(C^+(b-\ell)F(bC^-,\ell,\ell+1-bC^+;1-x )+ (\ell - b C^+) F(b C^-,\ell,\ell-bC^+;1-x) \right),
\end{multline*}
where $x= 1-\frac{U^-}{U^+}$ and $\Gamma(z)$ is Euler Gamma function. It is worth noting that both the coefficients $Q_{b,\ell}$ for $b,\ell \geq 1$ and $x$ depend on the Laplace variable $s$.

By using the contiguous relations satisfied by the  function $F(a,b,c;z)$ \cite{Abramowitz}, it is possible to show that 
$$
Q_{b,\ell} = -(U^-)^{\ell+1} \frac{\Gamma(b) \Gamma(1 - b C^+)}{\Gamma(b-bC^+)}\frac{x^{1-b}}{1-x} F(\ell-b,-bC^+,-b;x);
$$
see \cite{asymp} for details.

\subsubsection{Computation of $E(q,v)$}

By using the results of \cite{ridha}, we can now compute the unknown function $E(q,v)$ in terms of function $\Theta(s;w)$ defined as follows. For some $\nu\in \mathbbm{C}$: define $\theta(\nu;w)$ as the solution to the equation
\begin{equation}
    \label{deftheta}
    1-\theta+w \theta^{\nu} =0,
\end{equation}
analytic in the neighbourhood of $w=0$ and such that $\theta(\nu, 0)=1$; see \cite{Polya}. In the following, we set $\Theta(s;w)=\theta(1-C^+(s);w)$. We also introduce the function $\Sigma(\nu;w)$ defined by
\begin{equation}
\label{defSigma}
\Sigma(\nu,w) = \frac{w}{\theta(\nu,w)}\frac{\partial \theta}{\partial w}(\nu;w).
\end{equation}
It is easily checked that by taking the derivative of Equation~\eqref{deftheta}
\begin{equation}
    \label{sigmatheta}
  \Sigma(\nu;w) = \frac{ \theta(\nu;w) -1}{(1-\nu)  \theta(\nu;w) +\nu}.
\end{equation}

In addition, from\cite{Polya}, we have
\begin{equation}
\label{defSigmab}
\Sigma(\nu;w) = \sum_{\ell=1}^\infty \frac{\Gamma(\ell \nu)}{\Gamma(\ell)\Gamma(1-(1-\nu)\ell)} w^\ell
\end{equation}
defined for $|w|< |\exp(-\psi(\nu))|$, where
$$
\psi(\nu) =\left\{ \begin{array}{ll} (1-\nu)\log(1-\nu) + \nu\log(-\nu) & \nu \in \mathbbm{C}\setminus[0,\infty), \\   (1-\nu)\log(1-\nu) + \nu\log(\nu) & \nu \in [0,1], \\
 (1-\nu)\log(\nu - 1 ) + \nu\log(\nu) & \nu \geq 1.
\end{array}
\right.
$$

\begin{prop}
The function $E(q,v)$ is given for $|v|<|\exp(-\psi(1-C^+(s)))|$ by
\begin{equation}
    \label{eqvexp}
E(q,v) =  \frac{Q_0(v)}{(U^+-U^-)P(v)}  \int_0^{U^-}\Psi_0(\Theta(x R(\xi)X(v))) \frac{d \xi}{1-\xi},
\end{equation}
where 
\begin{eqnarray}
    \label{defQ0}
    Q_0(v) &=& U^+ U^--q v, \\
    \label{defX}
    X(v)& =&  \frac{v}{v-U^-}\left(\frac{1-\frac{v}{U^-}}{1-\frac{v}{U^+}}  \right)^{C^+} = \frac{-U^+ v}{P(v)R(v)},\\
    \label{defPsi0}
\Psi_0(t) &=& \frac{t(1-t)}{(C^+ t +1-C^+)^3}.
\end{eqnarray}
\end{prop}

\begin{proof}
By applying the results of \cite{ridha}, we have
\begin{multline*}
E(q,v) = \frac{-1}{U^+-U^-} \left( \frac{C^-}{1-\frac{v}{U^-}} + \frac{C^+}{1-\frac{v}{U^+}} \right)\\  \int_0^{U^-} \sum_{\ell=1}^\infty \frac{\Gamma(\ell(1-C^+))}{\Gamma(\ell)\Gamma(1-\ell C^+)} \ell  (x X(v) R(\xi))^\ell \frac{d\xi}{1-\xi},   
\end{multline*}
where $X(v)$ is defined by Equation~\eqref{defX}.

It is worth noting that 
$$
 \left( \frac{C^-}{1-\frac{v}{U^-}} + \frac{C^+}{1-\frac{v}{U^+}} \right) = \frac{U^+U^--q v}{P(v)} = v \frac{X'(v)}{X(v)},
$$
so that 
$$
E(q,v) = \frac{-v}{U^+-U^-}\frac{\partial\mathcal{E} }{\partial v}( v),
$$
where
$$
\mathcal{E}(v) = \int_0^{U^-} \sum_{\ell=1}^\infty \frac{\Gamma(\ell(1-C^+))}{\Gamma(\ell)\Gamma(1-\ell C^+)} (x R(\xi)X(v))^\ell   \frac{d\xi}{1-\xi} .
$$
 By using the series expansion~\eqref{defSigmab}, we obtain
\begin{align*}
    \mathcal{E}(v) &= \int_0^{U^-}  \Sigma (1-C^+(s);x R(\xi)X(v)) \frac{d\xi}{1-\xi} \\ &=  \int_0^{U^-} \frac{\Theta(x R(\xi)X(v))-1}{C^+\Theta(xR(\xi)X(v))+1-C^+}\frac{d \xi}{1-\xi},
\end{align*}
where we have used Equation~\eqref{sigmatheta}. By taking derivatives, the result follows. 
\end{proof}

To conclude this section, let us note  the following relation, which will be useful in the computation of $F(u,v)$:
\begin{equation}
    \label{derivePsi0}
    \frac{\partial \Psi_0}{\partial v}(\Theta(xR(\xi)X(v)) =-\frac{Q_0(v)}{vP(v)}\Psi_1(\Theta(xR(\xi)X(v)),
\end{equation}
where 
\begin{equation}
    \label{defPsi1}
 \Psi_1(x) = \frac{t(1-t)(1-2 t -C^+(1-t^2))}{(C^+ t +1-C^+)^5}.  
\end{equation}

\subsection{Computation of the function $F(u,v)$}

By using the expression of $E(q,v)$ we are able to compute the function $F(u,v)$.
\begin{prop}
The function $F(u,v)$ is given by
\begin{align}
F(u,v) &=\frac{u}{(u-v)P(u)} \int_u^{U^-} \left(1-Z\left(u,\frac{v}{u};y\right)\right) Z\left(u,\frac{v}{u};y\right)\frac{d y}{1-y} \label{Fuv} \\
& +\frac{u}{(u-v)P(u)}\int_u^{U^-} \left(1-Z\left(u,\frac{v}{u};y\right)\right) L_1\left(y, y Z\left(u,\frac{v}{u};y\right)   \right)     \frac{d y}{y} \nonumber \\
& +  \frac{u}{(u-v)P(u)}\int_u^{U^-}\left(1-Z\left(u,\frac{v}{u};y\right)\right) L_2\left(y Z\left(u,\frac{v}{u};y\right) \right)     \frac{d y}{y},\nonumber
\end{align}
 where
$$
L_1(u,v)=\frac{1}{U^+-U^-}  \left(u \frac{Q_0(v)}{P(v)}   + v \frac{Q_1(v)}{P(v)^2} \right)   \int_0^{U^-}\Psi_0(\Theta(x R(\xi)X(v))) \frac{d \xi}{1-\xi} 
$$
and
$$
 L_2(v) = \\  \frac{(U^++U^--q-v)Q_0(v)^2}{(U^+-U^-)P(v)^2} \int_0^{U^-}  \Psi_1(\Theta(x R(\xi)X(v)))    \frac{d \xi}{1-\xi}
$$
with 
$$
Q_1(v) = (q^2-U^+U^-)v^2 + 2U^+U^-(U^++U^--2q)v - U^+U^-((U^++U^--q)^2-U^+U^-),
$$
and $\Psi_1(t)$ is defined by Equation~\eqref{defPsi1}.
\end{prop}

\begin{proof}
The function $L(u,v)$ is defined by
$$
 L(u,v) = \frac{v}{1-u}+ (u+v)E(q,v) -v(1+\rho+s-v)\frac{dE}{dv}(q,v).
 $$
By using the expression of $E(q,v)$ given by Equation~\eqref{eqvexp}, Equation~\eqref{Fuv} follows after some algebra, where we use Equation~\eqref{derivePsi0}.
\end{proof}

In the next section, we use the expression of $F(u,v)$ to express the Laplace transform of the batch sojourn time $\Omega$. The mean value however seems very difficult to compute from this Laplace transform. This is why we shall derive the mean value by using the same machinery used so far to compute $F(u,v)$ but for a function appearing in the series expansion of $F(u,v)$  for $s$ in the neighbourhood of $s=0$.

\section{Batch sojourn time}
\label{Sec:Batch}

\subsection{Laplace transform}

Let us first recall that in the $M^{X]}/M/1$-PS queue, the stationary distribution of the number $N$ of jobs is given by
\begin{equation}
    \label{statdis}
\mathbbm{P}(N=n) = \left(1-\frac{\rho}{1-q}\right) \rho (\rho+q)^{n-1}\mathbbm{1}_{\{n \geq 1\}} + \left(1-\frac{\rho}{1-q}\right) \mathbbm{1}_{\{n =0\}}.
\end{equation}

\begin{prop}
 The batch sojourn time $\Omega$ in the $M^{X]}/M/1$-PS queue has the Laplace transform  given for $\Re(s) \geq 0$ by
\begin{equation}
\label{laptransfoW}
\mathbbm{E}(e^{-s \Omega}) = \frac{1-\rho-q}{q(\rho+q)}\left(\rho^2 F(s;\rho+q,q) + \frac{q^3+\rho(q s +\rho+2q(1-q))E(s;q,q)}{q +\rho+ qs -q^2}   \right),
\end{equation}
where the functions $F(s;u,v)$ and $E(s;q,q)$ are defined by Equations~\eqref{Fuv} and \eqref{eqvexp}, respectively.
\end{prop}

\begin{proof}
By using the stationary probability distribution given by Equation~\eqref{statdis}, we have
$$
\mathbbm{E}(e^{-s \Omega}) = \frac{1-\rho-q}{q(\rho+q)}\left(\rho E(s;\rho+q,q) + q E(s;0,q)   \right)
$$
and by using the relation between $E(s;u,v)$ and $F(s;u,v)$, we have
$$
\mathbbm{E}(e^{-s \Omega}) = \frac{1-\rho-q}{q(\rho+q)}\left(\rho^2 F(s;\rho+q,q) -q^2 F(s;0,q)  +(\rho+q)E(s;q,q) \right)
$$

The function $F(s;0,v)$ can be computed from the differential equation~\eqref{PDEF0} for $u=0$ and we find
$$
F(s;0,v) = \frac{v+(v-s-\rho-1)E(s;q,v)}{q v -\rho - q - s q}
$$
so that
$$
F(s;0,q) = \frac{q+(q-s-\rho-1)E(s;q,q)}{q^2 -\rho - q - s q}.
$$
Equation~\eqref{laptransfoW} follows.
\end{proof}

\subsection{Mean values}

In this section, we assume that the Laplace variable $s=0$ so that
\begin{equation}
    \label{defP0u}
    P(u) = (1-u)(\rho+q -u), \; C^+=\frac{-\rho}{1-q-\rho}, \; C^-=\frac{1-q}{1-\rho-q},
\end{equation}
and
$$
R(\xi) = \frac{1}{1-\xi}\left(\frac{1-\xi}{1-\frac{\xi}{\rho+q}} \right)^{C^+}.
$$

To compute the mean value of $\Omega$, we use an expansion of $F(s;u,v)$ for small $s$. Since 
$$
F(s;u,v) = \sum_{n=0}^\infty \sum_{b=1}^\infty \E(e^{-s\Omega_{n,b}}) u^n v^b
$$
and for all $n\geq 0 $ and $b\geq 1$
$$
 \E(e^{-s\Omega_{n,b}}) = 1-s \E(\Omega_{n,b}) + R_{n,b}(s)
$$
with $\lim_{s\to 0} R_{n,b}(s) =0$, we can write
$$
F(s;u,v) = F^{(0)}(u,v) + s  F^{(1)}(u,v) + \mathcal{R}_1(F)(s;u,v)
$$
where $ F^{(1)}(u,v)$ and $\mathcal{R}_1(F)(s;u,v)$ are analytic in $\mathbbm{D}\times \mathbbm{D}$ and 
$$
\lim_{s\to 0}\mathcal{R}_1(F)(s;u,v)=0
$$
for all $(u,v) \in \mathbbm{D}\times \mathbbm{D}$. In the same way, $E(s;q,v)$ can be expanded as
 $$
E(s;q,v) = E^{(0)}(v) + s  E^{(1)}(v) + \mathcal{R}_1(E)(s;v)
$$
with 
$$
\lim_{s\to 0}\mathcal{R}_1(E)(s;v)=0
$$
for all $v\in \mathbbm{D}$.

\begin{lemma}
The function $F^{(1)}(u,v)$ satisfies the differential equation
\begin{multline}
 u P(u) \, \frac{\partial F^{(1)}}{\partial u}(u,v) + 
v \left [ \rho(1-\QI) - (1+\rho-v)(u-\QI) \right ] \, 
\frac{\partial F^{(1)}}{\partial v}(u,v) \\
+ \; \left [ u(u-1-\rho) + (u-\QI)(u+v) \right ] \ F^{(1)} (u,v)  +  L^{(1)}(u,v) = 0 
\label{PDEF1}
\end{multline}
{with polynomial $P(u)$ given by Equation~\eqref{defP0u} and}
\begin{equation}
L^{(1)}(u,v) =L^{(1)}_0(u,v) + (u+v) \, E^{(1)}(\QI,v) -v (1+\rho-v)  
\frac{\partial E^{(1)}}{\partial v}(\QI,v),
\label{defL1}
\end{equation}
the term $L^{(1)}_0(u,v)$ being defined as 
$$
L^{(1)}_0(u,v) =- \frac{v(1- u v)}{(1-u)^2(1-v)^2}= \sum_{b=1}^\infty k_b^{(1)} (u) v^b
$$
with
\begin{equation}
    \label{defkb}
 k^{(1)}_b(u) = -\frac{u}{(1-u)^2} - \frac{b}{1-u}.
 \end{equation}
\end{lemma}

\begin{proof}
We know that 
$$
E(0;u,v) = \frac{v}{(1-u)(1-v)}
$$
and then
$$
F^{(0)}(u,v) = \frac{v}{(1-q)(1-u)(1-v)} \mbox{ and } E^{(0)}(q,v) =\frac{v}{(1-q)(1-v)}.
$$
Replacing $F(u,v)$ by  $ F^{(0)}(u,v) + s  F^{(1)}(u,v) + \mathcal{R}_1(F)(s;u,v)$ in  PDE~\eqref{PDEF0} and letting $s$ tend to 0  yield Equation~\eqref{PDEF1}, with
\begin{eqnarray*}
L^{(1)}_0(u,v) &=& -v \frac{\partial E^{(0)}}{\partial v}(q,v) -u F^{(0)}(u,v) - u(u-q) \frac{\partial F^{(0)}}{\partial u} -v(u-q) \frac{\partial F^{(0)}}{\partial v}\\
&=&- \frac{v(1- u v)}{(1-u)^2(1-v)^2}= \sum_{b=1}^\infty k_b^{(1)} (u) v^b,
\end{eqnarray*}
where the coefficients $k_b^{(1)}$ are defined by Equation~\eqref{defkb}.
\end{proof}

From the differential equation~\eqref{PDEF1}, we can use the analyticity condition at point 0 to compute function $E^{(1)}(q,v)$.

\begin{corol}
The function $E^{(1)}(q,v)$ is given by
\begin{multline}
E^{(1)}(q,v) = 
    \label{e1qvexp}
 \frac{Q_0(v)}{(U^+-U^-)P(v)}  \\  \int_0^{\rho+q}\left(  \Psi_0(\Theta(x (1-\xi)  R(\xi)X(v))) - \Psi_0(\Theta(x R(\xi)X(v)))    \right) \frac{(q-\xi)d \xi}{(\rho+q-\xi)(1-\xi)^2},
\end{multline}
where the function $\Psi_0(t)$ is defined by Equation~\eqref{defPsi0} and polynomial $Q_0(v)$ by Equation~\eqref{defQ0} and where it is implicitly assumed that the Laplace variable $s=0$.
\end{corol}

\begin{proof}
We can write 
$$
F^{(1)}(u,v) = \sum_{b=1}^\infty F_b^{(1)}(u) v^b.
$$
By using the same arguments as for the proof of Lemma~\ref{lemFb}, we have
$$
F_b^{(1)}(u) =  \\ \int_u^{\rho+q} f_b^{(1)}(z)    \frac{ R(z)^{b}z^{b-1}}{u^b \, P(u) R(u)^b}   \, \mathrm{d}z,
$$
where we set
$$
f_b^{(1)}(z) = k^{(1)}_b(z) + b E^{(1)}_{b-1}(z) \mathbbm{1}_{\{b \geqslant 2\}} -  \left( b(1+\rho+s) -z \right)  E_b^{(1)}(\QI)  .
$$

The analyticity condition of this function at point 0 requires 
\begin{equation}
    \label{condanalytic2}
     \int_0^{\rho+q}f_b^{(1)}(z){ R(z)^{b}z^{b-1}}  \, \mathrm{d}z =0.
\end{equation}
By using the same arguments as in the proof of Lemma~\ref{fb2}, we have for $k<b$
\begin{multline}
\label{rectech2}
  \int_0^{\rho+q} E^{(1)}_{b-k}(z)R(z)^{b}z^{b-k}  dz  = -\frac{(b-k)}{k}  \int_0^{\rho+q} E^{(1)}_{b-k-1}(z)R(z)^{b}z^{b-k-1}  dz \\ +  \frac{E^{(1)}_{b-k}(q)}{k}  \int_0^{\rho+q} ((b-k)(1+\rho+s)+(k-1) z)   R(z)^{b}z^{b-k-1}  dz  \\  -\frac{1}{k}\int_0^{\rho+q} k^{(1)}_{b-k}(z) R(z)^{b}z^{b-k-1}  dz       .
\end{multline}
Equation~\eqref{condanalytic2} implies that for $b \geq 2$
\begin{multline*}
E^{(1)}_b(q) \int_0^{\rho+q} \left( b(1+\rho+s) -z \right)    R(z)^{b}z^{b-1}  dz  = b  \int_0^{\rho+q} E^{(1)}_{b-1}(z)R(z)^{b}z^{b-1}  dz \\
+\int_0^{\rho+q} k_b^{(1)}(z)R(z)^b z^{b-1}dz
\end{multline*}
and by iterating Equation~\eqref{rectech2}, we obtain for $b \geq 1$
$$
-\sum_{\ell=1}^{b} (-1)^\ell \binom{b}{\ell}   Q_{b,\ell} E^{(1)}_\ell(q) \\
= \sum_{\ell=1}^b  (-1)^\ell \binom{b}{\ell}      \int_0^{\rho+q}k^{(1)}_\ell(z) {R(z)^b}z^{\ell-1}dz,
$$
where the coefficients $Q_{b,\ell}$ are defined by Equation~\eqref{defQbl} (with $s=0$ in the present case).

By using the definition~\eqref{defkb} of the coefficients $k^{(1)}_b(u)$, we have
\begin{multline*}
 \sum_{\ell=1}^b  (-1)^\ell \binom{b}{\ell}      \int_0^{\rho+q}k^{(1)}_\ell(z) {R(z)^b}z^{\ell-1}dz \\ =  \int_0^{\rho+q} \frac{(1-z)^b(b-1)+1}{(1-z)^{2}} R(z)^b dz
 =  b    \int_0^{\rho + q} \frac{(q-z)\left( (1-z)^b-1\right)}{(\rho+q-z)(1-z)^2}   {R(z)^b} dz
 \end{multline*}
via an integration by part. We eventually find
$$
- \sum_{\ell=1}^{b} (-1)^\ell \binom{b}{\ell}   Q_{b,\ell} E^{(1)}_\ell(q) \\
= b    \int_0^{\rho + q} \frac{(q-z)\left( (1-z)^b-1\right)}{(\rho+q-z)(1-z)^2}   {R(z)^b} dz,
$$
which can be rewritten as
\begin{multline*}
 \sum_{\ell=1}^{b} (-1)^\ell \binom{b}{\ell} F(\ell-b,-b C^+,-b;x)(U^-)^{\ell+1}  E^{(1)}_\ell(q) = \\
 \frac{b \Gamma(b-bC^+)}{\Gamma(b)\Gamma(1-b C^+)}\frac{1-x}{x} x^b  \int_0^{\rho + q} \frac{(q-z)\left( (1-z)^b-1\right)}{(\rho+q-z)(1-z)^2}   {R(z)^b} dz.
 \end{multline*}
Now, by using the same arguments as for the derivation of $E(q,v)$, notably the results in \cite{ridha}, we obtain Equation~\eqref{e1qvexp}.
\end{proof}

By using PDE~\eqref{PDEF1}, we  can  compute $F^{(1)}(\rho+q,v)$.

\begin{corol}
The function $F^{(1)}(\rho+q,v)$ is given by
\begin{multline}
    \label{defF1rhoq}
 F^{(1)} (\rho+q,v) =\frac{v}{(1-q-\rho )^2}\left(1-\frac{v}{q+\rho }\right)^{\frac{1-q-2 \rho }{\rho }} 
\Omega_1(v) + \Omega_2(v)  \\  -\frac{1-v+\rho}{\rho(q-v+\rho)}E^{(1)}(q,v),
\end{multline}
 where 
 \begin{multline}
 \label{defOmega1}
\Omega_1(v) =   \frac{v}{1-q+\rho } F_1\left(\frac{1-q+\rho }{\rho },\frac{1-q-\rho }{\rho },2,\frac{1-q+2\rho}{\rho },\frac{v}{q+\rho },v\right) \\
-
 \frac{(1-v)^{\frac{-1+q}{\rho }}}{(1-q) (q+\rho )} {}_2F_1\left(\frac{1-q}{\rho },\frac{1-q-\rho }{\rho },\frac{1-q+\rho }{\rho },\frac{v (1-q-\rho )}{(1-v) (q+\rho )}\right),
 \end{multline}
 and 
 \begin{multline}
 \label{defOmega2}
 \Omega_2(v) = \\
   \frac{(1-q)(\rho+q)}{\rho^2}v^{\frac{\rho+q-1}{\rho}}(\rho+q-v)^{-\frac{2\rho+q-1}{\rho}} \int_0^v (1-\xi) \xi^{\frac{1-2\rho-q}{\rho}}(\rho+q-\xi)^{-\frac{1-q}{\rho}} E^{(1)}(q,\xi)  d\xi,
 \end{multline}
 $F_1(\alpha,\beta,\beta',\gamma,x,x')$ and ${}_2F_1 (a,b,c,z)$ denoting Appell function of the first kind \cite{schloesser} and  Gauss Hypergeometric function \cite{Abramowitz}, respectively.
\end{corol}

\begin{proof}
Setting $u=\rho+q$ in PDE~\eqref{PDEF1}, we obtain since $P(\rho+q)=0$
\begin{multline*}
   -\rho v (\rho+q-v ) \, 
\frac{dF^{(1)}}{\partial v}(\rho+q,v) \\
+ \; \left( (\rho+q)(\rho+q-1) +\rho v \right) \ F^{(1)} (\rho+q,v)  +  L^{(1)}(\rho+q,v) = 0 .
\end{multline*}
A direct integration yields for $|v|<\rho+q$,
\begin{multline*}
 F^{(1)} (\rho+q,v) = \kappa v^{\frac{\rho+q-1}{\rho}}(\rho+q-v)^{-\frac{2\rho+q-1}{\rho}} \\ +\frac{1}{\rho}v^{\frac{\rho+q-1}{\rho}}(\rho+q-v)^{-\frac{2\rho+q-1}{\rho}} \int_0^v \xi^{\frac{1-2\rho-q}{\rho}}(\rho+q-\xi)^{-\frac{1-\rho-q}{\rho}} L^{(1)}(\rho+q,\xi) d\xi
\end{multline*}
for some integration constant $\kappa$. Because the function $F^{(1)}(\rho+q,v)$ has to be analytic in the neighbourhood of $v=0$ and since $\rho+q<1$, $\kappa$ must be null and then 
\begin{multline*}
 F^{(1)} (\rho+q,v) = \frac{1}{\rho}v^{\frac{\rho+q-1}{\rho}}(\rho+q-v)^{-\frac{2\rho+q-1}{\rho}} \\  \left(  \int_0^v \xi^{\frac{1-2\rho-q}{\rho}}(\rho+q-\xi)^{-\frac{1-\rho-q}{\rho}} L^{(1)}_0(\rho+q,\xi) d\xi -  \int_0^v f_1(\xi)  \frac{\partial E^{(1)}}{\partial v}(q,\xi)  d\xi \right. 
\\ \left. +  \int_0^v (\rho+q+\xi) \xi^{\frac{1-2\rho-q}{\rho}}(\rho+q-\xi)^{-\frac{1-\rho-q}{\rho}} E^{(1)}(q,\xi)  d\xi \right),
\end{multline*}
where
$$
f_1(\xi) = (1+\rho-\xi)\xi^{\frac{1-\rho-q}{\rho}}(\rho+q-\xi)^{-\frac{1-\rho-q}{\rho}}.
$$
An integration by parts then yields
\begin{multline*}
 F^{(1)} (\rho+q,v) = \frac{1}{\rho}v^{\frac{\rho+q-1}{\rho}}(\rho+q-v)^{-\frac{2\rho+q-1}{\rho}} \\  \left(  \int_0^v \xi^{\frac{1-2\rho-q}{\rho}}(\rho+q-\xi)^{-\frac{1-\rho-q}{\rho}} L^{(1)}_0(\rho+q,\xi) d\xi - f_1(v) E^{(1)}(q,v) \right. 
\\ \left. + \frac{(1-q)(\rho+q)}{\rho} \int_0^v (1-\xi) \xi^{\frac{1-2\rho-q}{\rho}}(\rho+q-\xi)^{-\frac{1-q}{\rho}} E^{(1)}(q,\xi)  d\xi \right),
\end{multline*}

The term
\begin{multline*}
\frac{1}{\rho}v^{\frac{\rho+q-1}{\rho}}(\rho+q-v)^{-\frac{2\rho+q-1}{\rho}}   \int_0^v \xi^{\frac{1-2\rho-q}{\rho}}(\rho+q-\xi)^{-\frac{1-\rho-q}{\rho}} L^{(1)}_0(\rho+q,\xi) d\xi =  \\ 
-\frac{1}{\rho(1-\rho-q)^2}v^{\frac{\rho+q-1}{\rho}}(\rho+q-v)^{-\frac{2\rho+q-1}{\rho}}   \int_0^v \xi^{\frac{1-\rho-q}{\rho}}(\rho+q-\xi)^{-\frac{1-\rho-q}{\rho}}\frac{1-(\rho+q)\xi}{(1-\xi)^2} d\xi
\end{multline*}
can be written as
\begin{multline*}
   \frac{1}{(1-\rho-q)^2} \frac{v}{\rho+q}\left(1-\frac{v}{\rho+q}\right) ^{\frac{1-2\rho-q}{\rho}}  \\ \left(\frac{v(\rho+q)}{1-q+\rho} F_1\left(\frac{1-q+\rho }{\rho },\frac{1-q-\rho }{\rho },2,\frac{1+2\rho -q}{\rho },\frac{v}{q+\rho },v\right)  \right. \\
\left.    -\frac{1}{1-q} F_1\left(\frac{1-q }{\rho },\frac{1-q-\rho }{\rho },2,\frac{1+\rho -q}{\rho },\frac{v}{q+\rho },v\right)\right).
\end{multline*}
 By using identities satisfied by Appell functions \cite{schloesser}, the above quantity is equal to 
\begin{multline*}
\frac{v}{(1-q-\rho )^2}\left(1-\frac{v}{q+\rho }\right)^{\frac{1-q-2 \rho }{\rho }} \\
\left(\frac{v}{1-q+\rho } F_1\left(\frac{1-q+\rho }{\rho },\frac{1-q-\rho }{\rho },2,\frac{1-q+2\rho}{\rho },\frac{v}{q+\rho },v\right) \right.\\
-
\left. \frac{(1-v)^{\frac{-1+q}{\rho }}}{(1-q) (q+\rho )}{}_2F_1 \left(\frac{1-q}{\rho },\frac{1-q-\rho }{\rho },\frac{1-q+\rho }{\rho },\frac{v (1-q-\rho )}{(1-v) (q+\rho )}\right)\right).
\end{multline*}
Equation~\eqref{defF1rhoq} then follows.
\end{proof}

By using the Laplace transform given by Equation~\eqref{laptransfoW}, we can compute the mean value   of random variable $\Omega$.

\begin{prop}
The mean value of the batch sojourn time $\Omega$ is given by
\begin{equation}
    \label{meanOmega}
    \E(\Omega) =  \frac{1-\rho-q}{q(q+\rho)} \left( \frac{(1-q)^2(\rho+q)}{\rho+q-q^2} E^{(1)}(q,q) + \frac{q^3}{(1-q)(\rho+q-q^2)} -  \omega \right),
\end{equation}
where $\omega=\rho^2( \Omega_1(q)+\Omega_2(q))$ with the functions $\Omega_1(v)$ and $\Omega_2(v)$ being defined by Equations~\eqref{defOmega1} and \eqref{defOmega2}.
\end{prop}

\begin{proof}
By using the series expansion of $F(u,v)$ and $E(q,v)$ in variable $s$ in Equation~\eqref{laptransfoW} and letting $s$ tend to 0, we obtain
\begin{multline*}
    \E(\Omega) = - \frac{1-\rho-q}{q(q+\rho)} \left(\rho^2 F^{(1)}(\rho+q,q)-\frac{q^3}{(1-q)(\rho+q-q^2)} \right. \\ \left. +\frac{\rho(\rho+2q(1-q))}{q+\rho-q^2} E^{(1)}(q,q)  \right),
\end{multline*}

By the expression~\eqref{defF1rhoq} of $F^{(1)}(\rho+q,v)$, we have for $v=q$
\begin{multline*}
 F^{(1)} (\rho+q,q) = -\frac{1-q+\rho}{\rho^2}E^{(1)}(q,q) + \frac{q}{(1-q-\rho )^2}\left(\frac{\rho}{q+\rho }\right)^{\frac{1-q-2 \rho }{\rho }} \\
\left(\frac{q}{1-q+\rho } F_1\left(\frac{1-q+\rho }{\rho },\frac{1-q-\rho }{\rho },2,\frac{1-q+2\rho}{\rho },\frac{q}{q+\rho },q\right) \right.\\
-
\left. \frac{(1-q)^{\frac{-1+q}{\rho }}}{(1-q) (q+\rho )} {}_2F_1\left(\frac{1-q}{\rho },\frac{1-q-\rho }{\rho },\frac{1-q+\rho }{\rho },\frac{q (1-q-\rho )}{(1-q) (q+\rho )}\right)\right)  \\
+  \frac{(1-q)(\rho+q)}{\rho^3}\left(\frac{\rho}{q}\right)^{\frac{1-q-\rho}{\rho}} \int_0^q (1-\xi) \xi^{\frac{1-2\rho-q}{\rho}}(\rho+q-\xi)^{-\frac{1-q}{\rho}} E^{(1)}(q,\xi)  d\xi.
\end{multline*}
Equation~\eqref{meanOmega} easily follows.
\end{proof}

In spite of the apparent complexity of the expression of $\E(\Omega)$, numerical values can be obtained  by using computing systems such as Mathematica.

\section{Conclusion}
\label{Sec:Conclusion}

By conditioning on the number of jobs in the system as well as the number of jobs in a tagged job, we have established recurrence relations between the conditional sojourn times of a tagged job. These recurrence relations have been used to establish a PDE for an associated bivariate generating function. All the complexity in the analysis comes from the fact that this PDE involves an unknown  generating function, whose coefficients satisfy a lower triangular linear system involving hypergeometric polynomials. The resolution of this linear system is performed in \cite{ridha}, which is by itself a contribution to the abundant literature on infinite triangular linear systems.

The resolution of the linear system allows us to compute the Laplace transform  of the sojourn time of a tagged batch and subsequently the mean value. The analysis can be continued to derive the tail of the sojourn time distribution \cite{asymp}.


\bibliographystyle{plain}
\bibliography{biblio}

\begin{thebibliography}{10}

\bibitem{Abramowitz}
M.~Abramowitz and I.~Stegun.
\newblock {\em Handbook of Mathematical Functions}.
\newblock Dover Publications, 1965.

\bibitem{ARN15}
V.I. Arnold and R.~Cooke.
\newblock {\em Lectures on Partial Differential Equations}.
\newblock Universitext. Springer Berlin Heidelberg, 2003.

\bibitem{ayesta}
K.~Avrachenkov, U.~Ayesta, and P.~Brown.
\newblock Batch arrival processor-sharing with application to multi-level
  processor-sharing scheduling.
\newblock {\em Queueing Systems}, 50(4):459 -- 480, 2005.

\bibitem{bansal}
Nikhil Bansal.
\newblock Analysis of the m/g/1 processor-sharing queue with bulk arrivals.
\newblock {\em Oper. Res. Lett.}, 31(5):401–405, September 2003.

\bibitem{Flatto}
L.~Flatto.
\newblock The waiting time distribution for the random order service {M/M/1}
  queue.
\newblock {\em The Annals of Apllied Probability}, 7(2):382 --409, 1997.

\bibitem{grossharris}
Donald Gross, John~F. Shortle, James~M. Thompson, and Carl~M. Harris.
\newblock {\em Fundamentals of Queueing Theory}.
\newblock Wiley-Interscience, USA, 4th edition, 2008.

\bibitem{itc31}
F.~{Guillemin}, V.~{Quintuna Rodriguez}, and A.~{Simonian}.
\newblock A processor-sharing model for the performance of virtualized network
  functions.
\newblock In {\em 2019 31st International Teletraffic Congress (ITC 31)}, pages
  10--18, 2019.

\bibitem{GQS}
Fabrice Guillemin, Veronika Karina~Quintuna Rodriguez, and Alain Simonian.
\newblock Sojourn time in a processor sharing queue with batch arrivals.
\newblock {\em Stochastic Models}, 34(3):322--361, 2018.

\bibitem{Klein0}
L.~Kleinrock.
\newblock {\em Queueing Systems}, volume~1.
\newblock Wiley, New York, 1976.

\bibitem{Klein71}
L~Kleinrock, RR~Muntz, and E~Rodemich.
\newblock The processor sharing queueing model for time shared systems with
  bulk arrivals.
\newblock {\em Networks}, 1971.

\bibitem{ridha}
Ridha Nasri, Alain Simonian, and Fabrice Guillemin.
\newblock An inversion formula with hypergeometric polynomials.
\newblock To appear in Integral Transforms And Special Functions, 2020.

\bibitem{Polya}
George P\'{o}lya and Gabor Szegö.
\newblock {\em Problems and Theorems in Analysis I}.
\newblock Springer, 1978.

\bibitem{veronica2}
Veronica~Karina Quintuna~Rodriguez and Fabrice Guillemin.
\newblock {Performance analysis of resource pooling for network function
  virtualization}.
\newblock In {\em Networking Conference}, November 2016.

\bibitem{cloudran}
Veronica~Quintuna Rodriguez and Fabrice Guillemin.
\newblock Cloud-ran modeling based on parallel processing.
\newblock {\em {IEEE} Journal on Selected Areas in Communications},
  36(3):457--468, 2018.

\bibitem{schloesser}
M.~J. Schlosser.
\newblock Multiple hypergeometric series- appell series and beyond.
\newblock arXiv: 1305.1966v1, 2013.

\bibitem{asymp}
Alain {Simonian}, Ridha Nasri, Fabrice {Guillemin}, and Veronica {Quintuna
  Rodriguez}.
\newblock Asymptotic analysis of the sojourn time of a batch in an
  {$M^{[X]}/M/1$} {P}rocessor {S}haring queue.
\newblock Submitted for publication, 2020.

\end{thebibliography}
\appendix

\section{Proof of Proposition~\ref{CNCL}}
\label{App3}
Define the generating functions $E_b(u)$ for $b \geq 1$ and $|u|<1$ by
\begin{equation}
E_b(s,u) = \sum_{n=0}^\infty e_{n,b}^*(s) u^n
\qquad s > 0,
\label{defLaplEb}
\end{equation}
and the related functions
\begin{equation}
F_b(s,u) = 
\left\{
\begin{array}{ll}
\displaystyle \frac{E_b(s,u) - E_b(s,\QI)}{u-\QI}, \quad \quad 
s \geqslant 0, \; u \in \mathbb{D} \setminus \{\QI\},
\\ \\
\displaystyle \displaystyle \frac{\partial E_b}{\partial u}(s,\QI), \quad \quad \quad \quad \quad \quad
s \geqslant 0, \; u = \QI.
\end{array} \right.
\label{defFbTheta}
\end{equation}
In the rest of this section, we omit to specify the Laplace variable $s$ as argument of functions.

\begin{lemma}
\label{lemFb}
{For $b \geqslant 1$, function $F_b^*$ can be expressed in terms of 
$E_{b-1}^*$ by}
\begin{multline}
F_b(u) = \\ \int_u^{U^-} \left( \frac{ \mathbbm{1}_{\{b=1\}}}{1-z} + b E_{b-1}(z) \mathbbm{1}_{\{b \geqslant 2\}} -  \left( b(1+\rho+s) -z \right)  E_b(\QI)    \right)    \frac{ R(z)^{b}z^{b-1}}{u^b \, P(u) R(u)^b}   \, \mathrm{d}z
\label{Fb*}
\end{multline}
{for $u \in \mathbb{D} \setminus \{0\}$ and where $R(u)$  is defined by Equation~(\ref{defR})}.
\label{ExpressFb*}
\end{lemma}

\begin{proof}
From the recurrence~\eqref{E*}, we have for $b=1$,
\begin{eqnarray*}
 (s+\rho+1) \sum_{n=0}^\infty (n+1)e^*_{n,1}(s) u^n &=& \frac{1}{1-u}  + \sum_{n=0}^\infty n e^*_{n-1,1}(s) u^n \\ 
& +&(1-q)\rho  \sum_{n=0}^\infty (n+1)\sum_{m=1}^\infty q^{m-1}e^*_{n+m,1}(s) u^n
  \end{eqnarray*}
and for $b \geq 1$
\begin{eqnarray*}
 (s+\rho+1) \sum_{n=0}^\infty (n+b)e^*_{n,b}(s) u^n &= & b  \sum_{n=0}^\infty e^*_{n,b-1}(s) u^n  + \sum_{n=0}^\infty n e^*_{n-1,b}(s) u^n \\ 
 &+& (1-q)\rho  \sum_{n=0}^\infty (n+b)\sum_{m=1}^\infty q^{m-1}e^*_{n+m,b}(s) u^n
  \end{eqnarray*}
By using identity~\eqref{idtech}, the first equation yields
\begin{multline*}
 \frac{u \, P(u)}{u-q} \, \frac{\partial E_1}{\partial u}(u) + 
\left(u-(1+\varrho+s)\right) E_1(u) +   (1-q)\rho \frac{E_1(u) -E_1(q)}{u-q} \\ -u\rho(1-q) \frac{E_1(u)-E_1(q)}{(u-q)^2} +\frac{1}{1-u} = 0
\end{multline*}
and the second one for $b \geq 2$
\begin{multline*}
 \frac{u  P(u)}{u-q} \, \frac{\partial E_b}{\partial u}(u) + 
\left(u-b(1+\varrho+s)\right) E_b(u) +   (1-q)\rho b \frac{E_b(u) -E_b(q)}{u-q} \\ -u\rho(1-q) \frac{E_b(u)-E_b(q)}{(u-q)^2} + b E_{b-1} (u) = 0.  
\end{multline*}

 Expressing each equation in terms of $F_b^*$ after (\ref{defFbTheta}) then cancels out all denominators in $1/(u-\QI)$ or $1/(u-\QI)^2$ and we obtain
\begin{multline}
u  P(u) \frac{\partial F_b}{\partial u}(u) + Q_b(u) F_b(u) =  \\
\left( b(1+\rho+s) -u \right)  E_b^*(\QI) - \frac{1}{1-u} \mathbbm{1}_{\{b=1\}} - b E_{b-1}^*(u) \mathbbm{1}_{\{b \geqslant 2\}} 
\label{FbL}
\end{multline}
after some algebraic reduction and where $Q_b(u)$ denotes the quadratic polynomial 
$$
Q_b(u) = P(u) + u^2 - b(1+\varrho+s)u + (b-1)(s\QI + \varrho + \QI).
$$

We now solve the first order differential equation (\ref{FbL}) for $F_b^*$; noting that 
$$
- \frac{Q_b(u)}{u  P(u)} = 
- \frac{b}{u} + \frac{b-1-b \, C^+}{u-U^+} + 
\frac{b-1-b \, C^-}{u-U^-}
$$
after standard algebra and the use of definition (\ref{C+-}) for constants 
$C^\pm$ together with the relation $C^+ + C^- = 1$, the homogeneous differential  associated with the ordinary differential equation~(\ref{FbL}) has the general solution $\phi_b(u)$ given by 
$$\phi_b(u) = \kappa_b  u^{-b}
\left(1-\frac{u}{U^+}\right)^{b-1-b \, C^+}\left(1-\frac{u}{U^-}\right)^{b-1-b \, C^-}$$
for any multiplicative constant $\kappa_b$; using the method of the variation of constants, the general solution to the full equation~(\ref{FbL}) is easily derived as
\begin{multline*}
F_b(u) =   \kappa^0_b u^{-b}
\left(1-\frac{u}{U^+}\right)^{b-1-b \, C^+}\left(1-\frac{u}{U^-}\right)^{b-1-b \, C^-} 
\\ + 
\int_u^{U^-} \left( \frac{ \mathbbm{1}_{\{b=1\}}}{1-z} + b E_{b-1}(z) \mathbbm{1}_{\{b \geqslant 2\}} -  \left( b(1+\rho+s) -z \right)  E_b(\QI)    \right)    \frac{ R(z)^{b}z^{b-1}}{u^b \, P(u) R(u)^b}   \, \mathrm{d}z
\end{multline*}
for all $u \in \mathbb{D}$ and some constant $\kappa_b^0$. Now, the analyticity of this solution $F_b^*$ at point $u = U^- \in \mathbb{D}$ requires $\kappa_b^0$ to be zero and expression~(\ref{Fb*}) follows. 
\end{proof}

Equation~\eqref{Fb*} is valid for $u \neq 0$ but we know by definition that this function must be analytic at point 0. This is possible if and only if
\begin{equation}
\label{condanalytic}
\int_0^{U^-} \left( \frac{ \mathbbm{1}_{\{b=1\}}}{1-z} + b E_{b-1}(z) \mathbbm{1}_{\{b \geqslant 2\}} -  \left( b(1+\rho+s) -z \right)  E_b(\QI)    \right)     R(z)^{b}z^{b-1}  \, \mathrm{d}z =0.
\end{equation}

By using Equation~\eqref{Fb*}, we can state the following lemma.

\begin{lemma}
\label{fb2}
For $k < b-1$, we have
\begin{multline}
\label{rectech}
    \int_0^{U^-} E_{b-k}(z)R(z)^{b}z^{b-k}  dz  = -\frac{(b-k)}{k}  \int_0^{U^-} E_{b-k-1}(z)R(z)^{b}z^{b-k-1}  dz \\ +  \frac{E_{b-k}(q)}{k}  \int_0^{U^-} ((b-k)(1+\rho+s)+(k-1) z)   R(z)^{b}z^{b-k-1}  dz.
\end{multline}
\end{lemma}

\begin{proof} 
By using the relation between $E_{b-k}(z)$  and $F_{b-k}(z)$ for $k<b-1$, we have
\begin{multline*}
    \int_0^{U^-} E_{b-k}(z)R(z)^{b}z^{b-k}  dz  = \\ \int_0^{U^-} (z-q)F_{b-k}(z)R(z)^{b}z^{b-k}  dz + E_{b-k}(q)  \int_0^{U^-} R(z)^{b}z^{b-k}  dz
\end{multline*}
By using the expression of $F_{b-k}(z)$ given by Equation~\eqref{Fb*}, we obtain
\begin{multline*}
 \int_0^{U^-} (z-q)F_{b-k}(z)R(z)^{b}z^{b-k}  dz =  \int_0^{U^-} \frac{(z-q)R(z)^k}{P(z)} \\ \int_z^{U^-} \left(( b -k) E_{b-k-1}(y ) - E_{b-k} (q)  \left( (b-k)(1+\rho+s) -y \right) \right)   R(y)^{b-k}y^{b-k-1}  dy dz.
\end{multline*}
From Equation~\eqref{derivR}, we have
$$
\int_0^{y} \frac{(z-q)R(z)^k}{P(z)} dz = \frac{1}{k}(1-R(y)^k).
$$
By Equation~\eqref{condanalytic}
\begin{multline*}
 \int_0^{U^-} (z-q)F_{b-k}(z)R(z)^{b}z^{b-k}  dz = \\ -\frac{1}{k} \int_0^{U^-}  \left(( b -k) E_{b-k-1}(y ) - E_{b-k} (q)  \left( (b-k)(1+\rho+s) -y \right) \right)  R(y)^{b}y^{b-k-1}  dy
\end{multline*}
and Equation~\eqref{rectech} follows.
\end{proof}

Equation~\eqref{condanalytic} implies that for $b \geq 2$
$$
E_b(q) \int_0^{U^-} \left( b(1+\rho+s) -z \right)    R(z)^{b}z^{b-1}  dz  = b  \int_0^{U^-} E_{b-1}(z)R(z)^{b}z^{b-1}  dz 
$$
and by iterating Equation~\eqref{rectech}, we obtain
\begin{multline*}
\sum_{k=0}^{b-1} (-1)^k \binom{b}{k} E_{b-k}\int_0^{U^-} ((b-k)(1+\rho+s)+(k-1)z)R(z)^b z^{b-k-1}dz \\
=(-1)^{b+1} \frac{b(b-1)\ldots 2}{2.3 \ldots(b-1)}\int_0^{U^-} {R(z)^b}\frac{dz}{1-z}.
\end{multline*}
This equation can be rewritten as Equation~\eqref{syslin}.
\end{document}